\documentclass[10pt]{amsart}

\usepackage[latin1]{inputenc}
\usepackage[T1]{fontenc}

\usepackage{bm}

\usepackage[theoremfont]{baskervillef}
\usepackage[baskerville, varg]{newtxmath}
\usepackage[scr=boondox,cal=boondoxo,bb=boondox,frak=euler]{mathalfa}

\usepackage{hyperref}
\usepackage{amssymb}
\usepackage{cleveref}

\usepackage[style=alphabetic,natbib=true,backend=biber,giveninits=true]{biblatex}

\usepackage{graphicx}
\usepackage{multirow}
\usepackage{enumitem}

\allowdisplaybreaks[4]

\makeatletter

\newsavebox\myboxA
\newsavebox\myboxB
\newlength\mylenA

\newcommand*\xoverline[2][0.75]{%
    \sbox{\myboxA}{$\m@th#2$}%
    \setbox\myboxB\null
    \ht\myboxB=\ht\myboxA%
    \dp\myboxB=\dp\myboxA%
    \wd\myboxB=#1\wd\myboxA
    \sbox\myboxB{$\m@th\overline{\copy\myboxB}$}
    \setlength\mylenA{\the\wd\myboxA}
    \addtolength\mylenA{-\the\wd\myboxB}%
    \ifdim\wd\myboxB<\wd\myboxA%
       \rlap{\hskip 0.5\mylenA\usebox\myboxB}{\usebox\myboxA}%
    \else
        \hskip -0.5\mylenA\rlap{\usebox\myboxA}{\hskip 0.5\mylenA\usebox\myboxB}%
    \fi}

\makeatother

\newcommand{\D}{\mathcal{D}}
\renewcommand{\H}{\mathcal{H}}
\newcommand{\M}{\mathcal{M}}

\newcommand{\CC}{\mathbb{C}}
\newcommand{\HH}{\mathbb{H}}
\newcommand{\NN}{\mathbb{N}}

\newcommand{\RR}{\mathbb{R}}
\newcommand{\ZZ}{\mathbb{Z}}

\newcommand{\hyp}{{\mathrm{hyp}}}

\newcommand{\Mod}{\mathrm{Mod}}

\renewcommand{\mod}{\mathbin{\mathrm{mod}}}

\newcommand{\Id}{\mathrm{Id}}

\let\Sigmaaux\Sigma
\renewcommand{\Sigma}{\mbox{\scalebox{1.05}{$\Sigmaaux$}}}

\newcommand{\Irr}{\mathrm{Irr}}

\newcommand{\Aut}{\mathrm{Aut}}
\newcommand{\End}{\mathrm{End}}
\newcommand{\Aff}{\mathrm{Aff}}

\newcommand{\Sp}{\mathrm{Sp}}
\newcommand{\SL}{\mathrm{SL}}
\newcommand{\SU}{\mathrm{SU}}
\newcommand{\SO}{\mathrm{SO}}

\renewcommand{\sp}{\mathfrak{sp}}
\renewcommand{\sl}{\mathfrak{sl}}
\newcommand{\su}{\mathfrak{su}}
\newcommand{\so}{\mathfrak{so}}

\newcommand{\Sym}{\mathrm{Sym}}
\newcommand{\Sq}{\mathrm{Sq}}

\newtheorem{thm}{Theorem}[section]

\newtheorem{lem}[thm]{Lemma}
\newtheorem{prop}[thm]{Proposition}

\theoremstyle{definition}

\theoremstyle{remark}
\newtheorem{rem}[thm]{Remark}

\DeclareFieldFormat{title}{\mkbibquote{#1}}

\begin{document}

	\title{A family of quaternionic monodromy groups of the Kontsevich--Zorich cocycle}
	
	\author{Rodolfo Gutiérrez-Romo}
	\address{Institut de Mathématiques de Jussieu -- Paris Rive Gauche, UMR 7586, Bâtiment Sophie Germain, 75205 PARIS Cedex 13, France.}
	\email{rodolfo.gutierrez@imj-prg.fr}
	
	\begin{abstract}
		For all $d$ belonging to a density-$1/8$ subset of the natural numbers, we give an example of a square-tiled surface conjecturally realizing the group $\SO^*(2d)$ in its standard representation as the Zariski-closure of a factor of its monodromy. We prove that this conjecture holds for the first elements of this subset, showing that the group $\SO^*(2d)$ is realizable for every $11 \leq d \leq 299$ such that $d = 3 \bmod 8$, except possibly for $d = 35$ and $d = 203$.
	\end{abstract}
	
	\maketitle
	\markboth{R. GUTIÉRREZ-ROMO}{QUATERNIONIC MONODROMY GROUPS OF THE K--Z COCYCLE}
	
	\section{Introduction}
	
	A \emph{translation surface} $(X, \omega)$ is a compact Riemann surface equipped with a non-zero Abelian differential. Away from its zeroes, $\omega$ induces an atlas on $X$ all whose changes of coordinates are translations, called a \emph{translation atlas}. Translation surfaces can be packed into a moduli space endowed with a natural $\SL(2, \RR)$-action, given by post-composing with the coordinate charts of the translation atlas. The geometric and dynamical properties of this action have been extensively studied. We refer the reader to the surveys by Forni--Matheus \cite{FM:intro}, Wright \cite{W:translation_surfaces} and Zorich \cite{Z:flat_surfaces} for excellent introductions to the subject.
	
	Integrating over $\omega$ provides coordinate charts for the moduli space of translation surfaces, called \emph{period coordinates}. An \emph{affine invariant manifold} $\M$ is an immersed connected suborbifold of the moduli space of translation surfaces which is locally defined by linear equations having real coefficients and zero constant terms in period coordinates. By the landmark work of Eskin--Mirzakhani \cite{EM:measures} and Eskin--Mirzakani--Mohammadi \cite{EMM:orbits}, affine invariant manifolds coincide with orbit closures of the $\SL(2, \RR)$-action.
	
	The \emph{Hodge bundle} is the vector bundle over $\M$ whose fibres are the homology groups $H_1(M; \RR)$, where $M$ is the underlying topological surface of the elements of $\M$. The Gauss--Manin connection provides a natural way to compare fibres of the Hodge bundle. The \emph{Kontsevich--Zorich cocycle} over $\M$ is the dynamical cocycle over the Hodge bundle induced by the $\SL(2, \RR)$-action. This cocycle is flat for the Gauss--Manin connection.
	
	The \emph{monodromy group} of $\M$ is the group arising from the action of the (orbifold) fundamental group of $\M$ on the Hodge bundle. These groups can be also defined by the action of the Kontsevich--Zorich cocycle on an $\SL(2, \RR)$-invariant subbundle of the Hodge bundle. Moreover, the Kontsevich--Zorich cocycle is semisimple and its decomposition respects the Hodge structure \cite{F:semisimplicity}. Using these facts, Filip \cite{F:zero_exponents} showed that the possible Zariski-closures of the monodromy groups arising from $\SL(2, \RR)$-(strongly-)irreducible subbundles, at the level of real Lie algebra representations and up to compact factors, belong to the following list:
	
	\begin{enumerate}[label=(\roman{*})]
		\item $\sp(2g, \RR)$ in the standard representation;
		\item $\su(p, q)$ in the standard representation;
		\item $\su(p, 1)$ in an exterior power representation;
		\item $\so^*(2d)$ in the standard representation; or
		\item $\so_\RR(n,2)$ in a spin representation.
	\end{enumerate}
	
	Nevertheless, it is not known whether every Lie algebra representation in this list is realizable as a monodromy group \cite[Question 1.5]{F:zero_exponents}. Indeed, it is well-known that every group in the first item is realizable. The groups in the second item were shown to be realizable by Avila--Matheus--Yoccoz \cite{AMY:covers}. Moreover, the group $\SO^*(6)$ in its standard representation (which coincides with $\SU(3,1)$ in its second exterior power representation) is also realizable by the work of Filip--Forni--Matheus \cite{FFM:quaternionic}.
	
	\smallbreak
	The main theorem of this article is the following:
	
	\begin{thm}
		\label{thm:main}
		For each $d$ belonging to a density-$1/8$ subset of the natural numbers, there exists a square-tiled surface conjecturally realizing the group $\SO^*(2d)$ as the monodromy group of an $\SL(2,\RR)$-(strongly-)irreducible piece of its Kontsevich--Zorich cocycle. This conjecture depends on certain linear-algebraic conditions, which can be computationally shown to be true for small values of $d$. In this way, we show that $\SO^*(2d)$ is realizable for every $11 \leq d \leq 299$ in the congruence class $d = 3 \bmod 8$, except possibly for $d = 35$ and $d = 203$.
	\end{thm}
	\smallbreak
	
	Indeed, as was done by Filip--Forni--Matheus \cite{FFM:quaternionic}, we will show that these groups seem to arise in quaternionic covers of simple square-tiled surfaces.
	
	This article is organized as follows. In \Cref{sec:preliminaries}, we cover the required background on monodromy groups of square-tiled surfaces.  \Cref{sec:construction}	 shows the construction of the explicit family of square-tiled surfaces arising as quaternionic covers. Finally, we compute the desired monodromy groups in \Cref{sec:computation}.
	
	\section{Preliminaries}
	
	\label{sec:preliminaries}
	
	\subsection{Monodromy groups}
	
	Monodromy groups are, in general, a way to encode how a space relates to its universal cover. In the case of the Kontsevich--Zorich cocycle, they are defined as follows: given an affine invariant manifold (or orbit closure) $\M$, we define its monodromy group as the image of the natural map $\pi_1(\M) \to \Sp(H_1(M; \RR))$, where $M$ is an underlying topological surface and $\pi_1(\M)$ is the orbifold fundamental group. This means that they encode the homological action of the mapping classes identifying different points of the Teichmüller space to the same point of the moduli space.
		
	By the Hodge bundle we mean the vector bundle over $\M$ having $H_1(M; \RR)$ as the fibre over every point. The Kontsevich--Zorich cocycle is the dynamical cocycle defined over the Hodge bundle by the $\SL(2, \RR)$-action. An $\SL(2, \RR)$-invariant subbundle $E$ is a subbundle for which $g \cdot E_X = E_{g \cdot X}$ for every $X \in \M$ and $g \in \SL(2, \RR)$. A flat subbundle $E$ is a subbundle which is flat for the Gauss--Manin connection. Observe that a flat subbundle is necessarily $\SL(2, \RR)$-invariant, since if the curvature vanishes then the parallel transport is done along $\SL(2, \RR)$-orbits in the ``obvious'' way. The converse is not true in general: the flatness condition requires no curvature in every possible direction, including those which are not reachable by the $\SL(2, \RR)$-action.
		
	The Hodge bundle can be decomposed into irreducible pieces and monodromy groups can be defined for such pieces. One then has the following \cite[Theorem 1.1]{F:zero_exponents}:
	
	\begin{thm} \label{thm:zero_lyap_constraints}
		Let $E$ be a strongly irreducible flat subbundle of the Hodge bundle over some affine invariant manifold $\M$. Then, the presence of zero Lyapunov exponents implies that the Zariski-closure of the monodromy group has at most one non-compact factor, which, up to finite-index, is equal at the level of Lie group representations to:
		\begin{itemize}
			\item $\SU(p, q)$ in the standard representation;
			\item $\SU(p, 1)$ in any exterior power representation; or
			\item $\SO^*(2d)$ in the standard representation for some odd $d$.
		\end{itemize}
	\end{thm}
	
	Observe that this is a ``refined'' version of the constraints in the previous section, under stronger hypotheses.
	
	\subsection{Square-tiled surfaces}
	
	A square-tiled surface is a particular kind of translation surface defined as a finite cover of the unit square torus branched over a single point. That is, we say that a translation surface $(X, \omega)$ is \emph{square-tiled} if there exists a covering map $\pi\colon X \to \RR^2/\ZZ^2$, which is unramified away from $0 \in \RR^2/\ZZ^2$, and $\omega = \pi^*(dz)$, where $dz$ is the Abelian differential on $\RR^2/\ZZ^2$ induced by the natural identification $\RR^2 \simeq \CC$. We will often write $X$ to refer to $(X, \omega)$ for simplicity.
	
	Combinatorially, a square-tiled surface can be defined as a pair of \emph{horizontal} and \emph{vertical} permutations $h, v \in \Sym(\Sq(X))$, where $\Sq(X)$ is some finite set that we interpret as the \emph{squares} of $X$. These two permutations can be obtained from our original definition as the deck transformations induced respectively by the curves $t \mapsto (t,0)$ and $t \mapsto (0,t)$, with $t \in [0, 1]$, and the set of squares can be defined to be $\Sq(X) = \pi^{-1}((0,1)^2)$. Conversely, we can glue squares horizontally using $h$ and vertically using $v$ and define $\omega$ to be the pullback of $dz$ in each square to obtain a square-tiled surface as in the original definition.
	
	\subsubsection{\texorpdfstring{$\SL(2,\RR)$}{SL(2,R)}-action and monodromy groups}
	
	Every square-tiled surface $X$ is a Veech surface, that is, its $\SL(2,\RR)$-orbit is closed. In particular, this implies that any $\SL(2,\RR)$-invariant subbundle of the Hodge bundle over the orbit $\SL(2,\RR) \cdot X$ is actually flat. Therefore, \Cref{thm:zero_lyap_constraints} can be applied for any $\SL(2,\RR)$-(strongly-)irreducible subbundle.
	
	We say that square-tiled surface $X$ is \emph{reduced} if the covering map $\pi$ cannot be factored through another non-trivial covering of the torus. In this case, the elements $g \in \SL(2, \RR)$ such that $g \cdot X$ is a square-tiled surface are exactly $\SL(2, \ZZ)$. It is often the case that we study the action of $\SL(2, \ZZ)$ on $X$ instead of the entire $\SL(2, \RR)$-action, since square-tiled surfaces can be represented in purely combinatorial terms. The \emph{Veech group} of $X$, usually denoted $\SL(X)$, is the subgroup of $\SL(2, \ZZ)$ stabilizing $X$. It is always an arithmetic subgroup of $\SL(2, \ZZ)$ and its index coincides with the cardinality of $\SL(2, \ZZ) \cdot X$. Every square-tiled surface that we will consider is reduced.
	
	A square-tiled surface may also have non-trivial automorphisms. In this case, the $\SL(2, \ZZ)$-action does not immediately induce a homological action on the Hodge bundle. Indeed, automorphisms are precisely the reason why orbit closures are, in general, orbifolds and not manifolds. More precisely, we define an \emph{affine homeomorphism} as an orientation preserving homeomorphism of $X$ whose local expressions (with respect to the translation atlas) are affine maps of $\RR^2$. We denote the group of affine homeomorphisms by $\Aff(X)$. We may extract the linear part of an affine homeomorphism to get a surjective homomorphism $\Aff(X) \to \SL(X)$. The kernel of this homomorphism is the group $\Aut(X)$ of automorphisms of $X$. This can be encoded in the form of a short exact sequence:
	\[
		1 \to \Aut(X) \to \Aff(X) \to \SL(X) \to 1.
	\]
	In other words, if $M$ is the underlying topological surface of $X$, then $\Aut(X)$ is precisely the subgroup of $\Mod(M)$ stabilizing a lift of $X$ to the Teichmüller space of translation surfaces. In this sense, it measures to which extent the $\Mod(M)$-action fails to be free at $X$. Automorphisms can also be defined combinatorically: they are the elements of $\Sym(\Sq(X))$ that commute with both $h$ and $v$. It is well-known that if $X$ has only one singularity, then it has no non-trivial automorphisms.
	
	The homology group $H_1(M; \RR)$ admits a splitting $H_1^{\mathrm{st}}(M) \oplus H_1^{(0)}(M)$ into symplectic and mutually symplectically orthogonal subspaces. The subspace $H_1^{\mathrm{st}}(M)$ is two-dimensional and is usually called the \emph{tautological plane}. It is spanned the following two cycles: the sum of all bottom horizontal sides of the squares of $X$ oriented rightwards, and the sum of all left vertical sides of the squares of $X$ oriented upwards. The subspace $H_1^{(0)}(M)$ consists of the \emph{zero-holonomy cycles}, that is, the cycles $c$ such that $\int_c \omega = 0$.
	
	Let $\rho \colon \Aff(X) \to \Sp(H_1(M; \RR))$ be the representation induced by the homological action of $\Aff(X)$. By restricting this representation to an invariant subspace, we obtain a monodromy representation of a subbundle of the Hodge bundle. We define the \emph{monodromy group} of this subbundle to be the image of this representation.
	
	The group $\rho(\Aff(X))$ preserves the splitting $H_1(M; \RR) = H_1^{\mathrm{st}}(M) \oplus H_1^{(0)}(M)$. Moreover, the space $H_1^{\mathrm{st}}(M)$ is also irreducible and its monodromy group is a finite-index subgroup of $\SL(2, \ZZ) = \Sp(2, \ZZ)$ which can be identified with $\SL(X)$. The subspace $H_1^{(0)}(M)$ is in general reducible. Therefore, understanding monodromy groups means understanding the irreducible pieces of $H_1^{(0)}(M)$ and the way $\rho(\Aff(X))$ acts on them.
	
	\subsubsection{Constraints for monodromy groups}
	
	\label{sec:constraints}
	
	Let $G = \Aut(X)$. The vector space $H_1(M; \RR)$ has a structure of a $G$-module induced by the representation $G \to \Sp(H_1(M; \RR))$. Since $G$ is a finite group, it possesses finitely many irreducible representations over $\RR$ which we denote $\Irr_\RR(G)$. The $G$-module $H_1(M; \RR)$ can be decomposed as a direct sum of irreducible representations. That is:
	\[
		H_1(M; \RR) = \bigoplus_{\alpha \in \Irr_\RR(G)} V_\alpha^{\oplus n_\alpha},
	\]
	where each $V_\alpha$ is an irreducible subspace of $H_1(M; \RR)$ on which $G$ acts as the representation $\alpha$.
	
	We can collect the same $G$-irreducible representations into the so-called \emph{isotypical components}. That is, let $W_\alpha = V_\alpha^{\oplus n_\alpha}$ and then:
	\[
		H_1(M; \RR) = \bigoplus_{\alpha \in \Irr_\RR(G)} W_\alpha.
	\]
	The group $\rho(\Aff(X))$ does not, a priori, respect this decomposition because a general affine homeomorphism may not commute with every automorphism. However, since $G$ is a finite group, there exists a finite-index subgroup of $\Aff_*(X) \leq \Aff(X)$ whose every element commutes with every element of $G$. Replacing $\Aff(X)$ by some finite-index subgroup preserves the Zariski-closure of the resulting monodromy group.
	
	Given an irreducible representation $\alpha$ of $G$, we may define an associative division algebra $D_\alpha$: the centralizer of $\alpha(G)$ inside $\End_\RR(V_\alpha)$. Up to isomorphism, there are three associative real division algebras:
	\begin{itemize}
		\item $D_\alpha \simeq \RR$, and $\alpha$ is said to be real;
		\item $D_\alpha \simeq \CC$, and $\alpha$ is said to be complex; or
		\item $D_\alpha \simeq \HH$, and $\alpha$ is said to be quaternionic.
	\end{itemize}
	
	The following theorem \cites[Section 3.7]{MYZ:origamis}{MYZ:corrigendum} relates these cases to constraints for monodromy groups:
	\begin{thm} \label{thm:contraints}
		The Zariski-closure of the group $\rho(\Aff_*(X))|_{W_\alpha}$ is contained in:
		\begin{itemize}
			\item $\Sp(2g_\alpha, \RR)$ if $\alpha$ is real;
			\item $\SU(p_\alpha, q_\alpha)$ if $\alpha$ is complex; or
			\item $\SO^*(2d_\alpha)$ if $\alpha$ is quaternionic.
		\end{itemize}
	\end{thm}
	
	We will exploit these constraints to find the desired groups.
	
	\section{Construction of the family of square-tiled surfaces}
	
	\label{sec:construction}
	
	In this section, we will construct the quaternionic covers that realize the desired groups as the Zariski-closure of the monodromy of a specific flat irreducible subbundle of the Hodge bundle.
	
	Let $d \geq 3$ be an odd integer. We consider a ``staircase'' $X^{(d)}$ with $d$ squares: the square-tiled surface induced by the permutations $(2, 1)(4, 3)\dotsc(d-1,d-2)(d)$ and $(1)(2, 3)(4, 5)\dotsb(d-1, d)$. It belongs to the connected component $\H_{(d+1)/2}(d-1)^{\hyp}$. Its automorphism group is trivial, since it belongs to a minimal stratum.
	
	We construct a cover $\widetilde{X}^{(d)}$ of $X^{(d)}$ as follows: for each element $g$ of the quaternion group $Q = \{1, -1, i, -i, j, -j, k, -k\}$, we take a copy $X_g^{(d)}$ of $X^{(d)}$. We glue the $r$-th right vertical side of $X_g^{(d)}$ to the $r$-th left vertical side of $X_{gi}^{(d)}$. Similarly, we glue the $r$-th top horizontal side of $X_g^{(d)}$ to the $r$-th bottom horizontal side of $X_{gj}^{(d)}$. See \Cref{fig:covering}. This construction coincides, up to relabelling, with that of Filip--Forni--Matheus for $d = 3$ \cite[Section 5.1]{FFM:quaternionic}.
	
	\begin{figure}
		\includegraphics[width=0.7\textwidth]{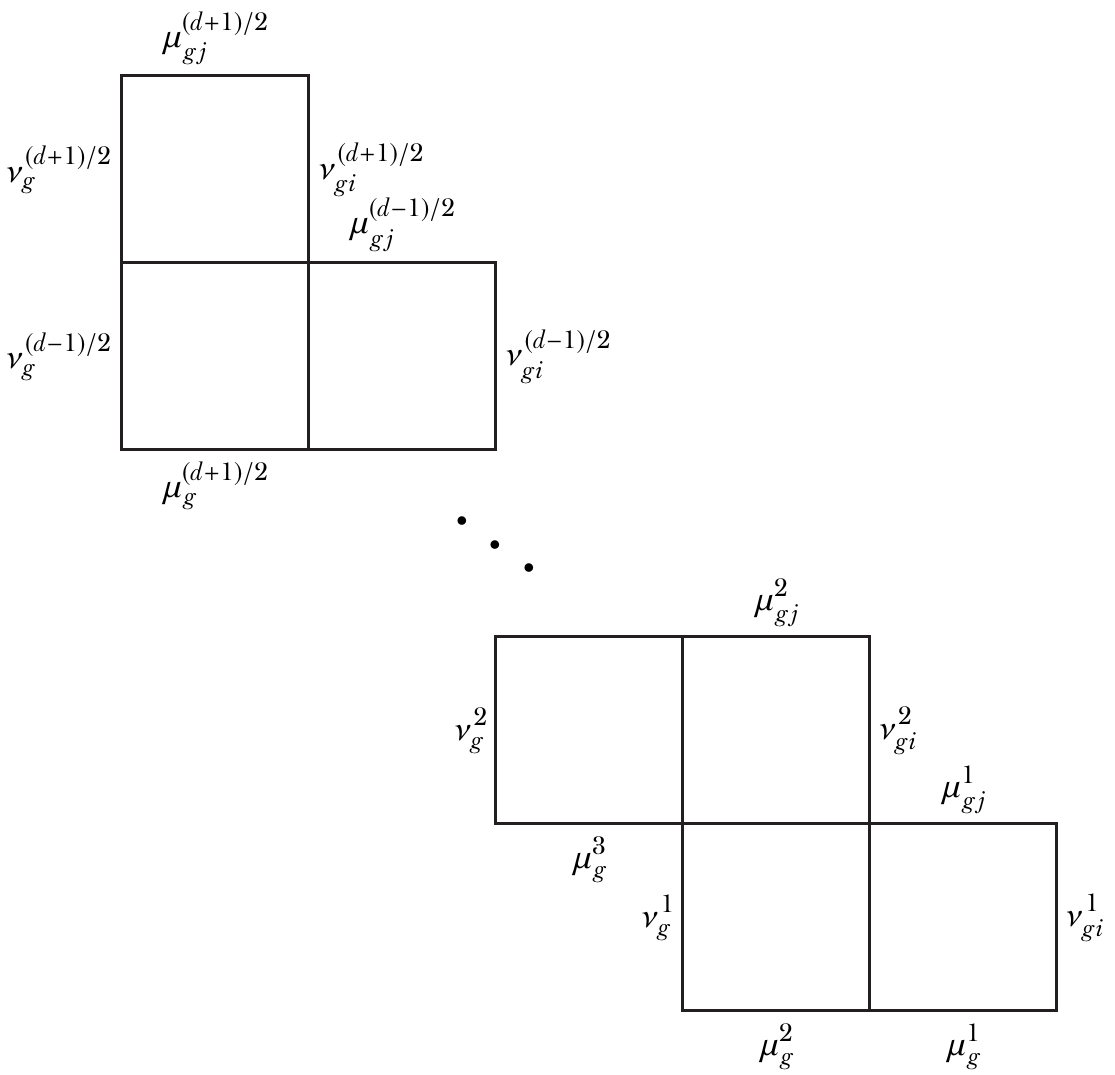}
		\caption{An illustration of $X_g^{(d)}$.}
		\label{fig:covering}
	\end{figure}

	For each $g \in Q$, we can define an automorphism $\varphi_g$ of $\widetilde{X}^{(d)}$ by mapping $X_h^{(d)}$ to $X_{gh}^{(d)}$ in the natural way, that is, preserving the covering map $\widetilde{X}^{(d)} \to X^{(d)}$ for each $h \in Q$. Indeed, the gluings are defined by multiplication on the right, which commutes with multiplication on the left. These are the only automorphisms of $\widetilde{X}^{(d)}$: an automorphism $\psi$ of $\widetilde{X}^{(d)}$ induces an automorphism of $X^{(d)}$ by ``forgetting the labels''. Since the only automorphism of $X^{(d)}$ is the identity, $X_1^{(d)}$ is mapped to some $X_g^{(d)}$ for $g \in Q$ in a way that preserves the covering map $\widetilde{X}^{(d)} \to X^{(d)}$. Thus, $\psi = \varphi_g$ and $\Aut(\widetilde{X}^{(d)}) \simeq Q$. We will denote $\Aut(\widetilde{X}^{(d)})$ by $G$.
	
	From now on, we will restrict to the case $d = 3 \mod 8$. The surface $\widetilde{X}^{(d)}$ has four singularities, each of order $2d-1$. Therefore, $\widetilde{X}^{(d)}$ belongs to the (connected) stratum $\H_{4d-1}((2d-1)^4)$.
	
	\begin{figure}
		\includegraphics[width=0.9\textwidth]{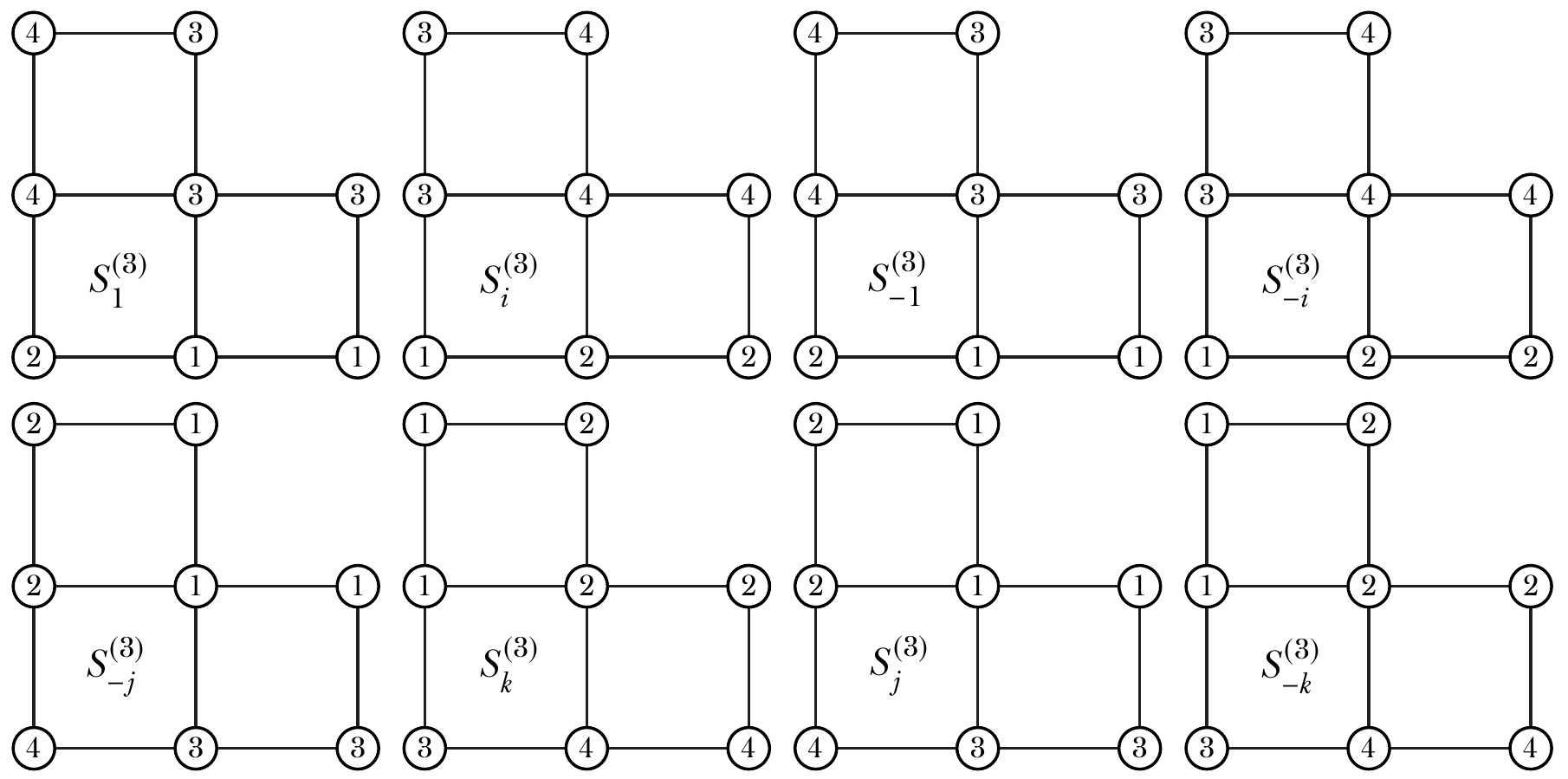}
		\caption{An illustration of $\widetilde{X}^{(3)}$ showing its four singularities.}
		\label{fig:stairs3}
	\end{figure}
	
	Since the automorphism $\varphi_{-1} \in G$ is an involution, it induces a splitting
	\[
		H_1(\widetilde{X}^{(d)}; \RR) = H_1^+(\widetilde{X}^{(d)}) \oplus H_1^-(\widetilde{X}^{(d)}),
	\]
	where $H_1^\pm(\widetilde{X}^{(d)})$ is the subspace of $H_1(\widetilde{X}^{(d)})$ where $\varphi_{-1}$ acts as $\pm \Id$. These subspaces are symplectic and symplectically orthogonal. The subspace $H_1^+(\widetilde{X}^{(d)})$ contains $H_1^{\mathrm{st}}(\widetilde{X}^{(d)})$ and is naturally isomorphic to $H_1(X_{\pm}^{(d)}; \RR)$, where $X_{\pm}^{(d)} = \widetilde{X}^{(d)}/\varphi_{-1}$. This latter surface is an intermediate cover of $X^{(d)}$ over the group $Q/\{1,-1\} \simeq \ZZ/2\ZZ \times \ZZ/2\ZZ$. Since every singularity of $\widetilde{X}^{(d)}$ is fixed by $\varphi_{-1}$, $X_{\pm}^{(d)}$ belongs to the stratum $\H_{2d-1}((d-1)^4)$. Therefore, $H_1^+(\widetilde{X}^{(d)})$ is a $(4d-2)$-dimensional subspace of the $(8d-2)$-dimensional space $H_1(\widetilde{X}^{(d)}; \RR)$ and we obtain that the dimension of $H_1^-(\widetilde{X}^{(d)})$ is $4d$.

	The irreducible representations (over $\CC$) of the group $Q$ can be summarized in the following character table:
	
	\begin{table}[ht!]
		\begin{tabular}{|c|c|c|c|c|c|c|}
			\hline & Dimension & $1$ & $-1$ & $\pm i$ & $\pm j$ & $\pm k$ \\
			\hline $\chi_1$ & $1$ & $1$ & $1$ & $1$ & $1$ & $1$ \\
			\hline $\chi_i$ & $1$ & $1$ & $1$ & $1$ & $-1$ & $-1$ \\ 
			\hline $\chi_j$ & $1$ & $1$ & $1$ & $-1$ & $1$ & $-1$ \\
			\hline $\chi_k$ & $1$ & $1$ & $1$ & $-1$ & $-1$ & $1$ \\
			\hline $\mathop{\mathrm{tr}} \chi_2$ & $2$ & $2$ & $-2$ & $0$ & $0$ & $0$ \\
			\hline 
		\end{tabular}
	\end{table}

	As detailed in \Cref{sec:constraints}, $H_1(\widetilde{X}^{(d)}; \RR)$ can be split into isotypical components associated with such representations. From the character table, we obtain that $H_1^-(\widetilde{X}^{(d)})$ corresponds to $2d$ copies of a $G$-irreducible representation whose character is the quaternionic character $\chi_2$, that is, $H_1^-(\widetilde{X}^{(d)}) = W_{\chi_2}$. Indeed, $\varphi_{-1}$ acts as the identity for any other representation in the table. We obtain the following:
	
	\begin{lem}\label{lem:zero_exp}
		The Zariski-closure of the monodromy group of the flat subbundle induced by $H_1^-(\widetilde{X}^{(d)})$ is a subgroup of $\SO^*(2d)$. Moreover, Kontsevich--Zorich cocycle over this subbundle has at least four zero Lypaunov exponents.
	\end{lem}
	
	\begin{proof}
		The first statement is a direct consequence of \Cref{thm:contraints}. The second statement is a consequence of the first since $d$ is odd  \cite[Corollary 5.5, Section 5.3.4]{F:zero_exponents}.
	\end{proof}
	
	We will prove that, for certain $d$ with $d \mod 8 = 3$, such Zariski-closure is actually $\SO^*(2d)$.
	
	\section{Computation of the monodromy groups}
	
	\label{sec:computation}
	
	\subsection{Dimensional constraints}
	
	In the presence of zero Lyapunov exponents, \Cref{thm:zero_lyap_constraints} states that the only possible Lie algebra representations of the Zariski-closure of the monodromy group of a flat subbundle are $\so^*(2d)$ in the standard representation, $\su(p, q)$ in the standard representation and $\su(p, 1)$ in some exterior power representation. The exterior power representations of $\su(p,1)$ are irreducible and faithful: by complexifying, one obtains $\sl(p+1, \CC)$ whose exterior power representations are known to satisfy these properties.
	
	Let
	\[
		\D = \left\{ d \in \NN \ \Big|\ 2d \neq \binom{p+1}{r} \text{ for every $p$ and $1 < r < p$} \right\}.
	\]
	The following lemma shows that, assuming irreducibility, $\so^*(2d)$ is the only possible Lie algebra for every $d \in \D$. By a slight abuse of notation, we will denote the standard representation of $\so^*(2d)$ by $\so^*(2d)$, the standard representation of $\su(p,q)$ by $\su(p,q)$ and the $r$-th exterior power of the standard representation of $\su(p,1)$ by $\Lambda^r(\su(p,1))$.
	
	\begin{lem}\label{lem:dim_constraints}
		Assume that $\so^*(2d)$, $\su(p,q)$ and $\Lambda^r(\su(p,1))$ act irreducibly on a vector space of the same dimension, that is, $2d = p + q = \binom{p+1}{r}$. Then, for every $d \in \D$ we have that $\dim_\RR \so^*(2d) < \dim_\RR \su(p,q)$ and that $\dim_\RR \so^*(2d) < \dim_\RR \Lambda^r (\su(p,1))$.
	\end{lem}
	
	\begin{proof}
		We have that
		\begin{align*}
			\dim_\RR \so^*(2d) &= d(2d-1) \\
			\dim_\RR \su(p,q) &= (p+q)^2 - 1 \\
			\dim_\RR \Lambda^r(\su(p,1)) &= p(p+2).
		\end{align*}
		
		Since $p+q = 2d$, $\dim_\RR \su(p,q) = 4d^2 - 1$. Therefore, $\dim_\RR \so^*(2d) < \dim_\RR \su(p,q)$ for every $d$. This argument also shows that $\dim_\RR \so^*(2d) < \dim_\RR \Lambda^r(\dim_\RR \su(p,1))$ for $r \in \{1, p\}$. Since $\Lambda^r (\su(p,1))$ acts on a $\binom{p+1}{r}$-dimensional space, this concludes the proof as $d \in \D$.
	\end{proof}
	
	\begin{rem}
		To obtain the strict inequality in the previous proof, it is necessary to assume that $d \in \D$. Indeed, if $d \notin \D$ then $2d = \binom{p+1}{r}$ with $1 < r < p$. This implies that $\dim_\RR \Lambda^r(\su(p,1)) \leq \dim_\RR \so^*(2d)$ since $2d = \binom{p+1}{r} \geq \binom{p+1}{2}$, so $p(p+1) \leq 4d$ and it is easy to check that this results in $p(p+2) \leq d(2d-1)$ if $d \geq 3$.
	\end{rem}
	
	\begin{lem}
		The set $\D$ has full density in $\NN$.
	\end{lem}
	
	\begin{proof}
		Let
		\[
			B_{p+1} = \left\{ \binom{p+1}{r} \ \mathrel{\Big|}\ 1 < r < p \right\} \text{ and } B = \bigcup_{p \geq 3} B_{p+1}.
		\]
		We will show that $|B \cap \{1, \dotsc, n\}| / n \to 0$. Observe that
		\[
			|B \cap \{1, \dotsc, n\}| \leq \sum_{p \geq 3} |B_{p+1} \cap \{1, \dotsc, n\}|
		\]
		Now, observe that:
		\begin{itemize}
			\item If $p \geq 3$ and $\binom{p+1}{2} > n$, then $|B_{p+1} \cap \{1, \dotsc, n\}| = 0$;
			\item If $p \geq 5$ and $\binom{p+1}{4} > n$, then $|B_{p+1} \cap \{1, \dotsc, n\}| \leq 2$. 
		\end{itemize}
		We will split the sum in this way to obtain a bound for $|B \cap \{1, \dotsc, n\}|$. Let $p_2$ be the smallest $p \geq 3$ such that $\binom{p+1}{2} > n$ and let $p_4$ be the smallest $p \geq 5$ such that $\binom{p+1}{4} > n$. We have that
		\begin{align*}
			|B \cap \{1, \dotsc, n\}| &\leq \sum_{p = 3}^{p_4 - 1} (p+1) + \sum_{p = p_4}^{p_2 - 1} 2 \leq p_4(p_4-1) + 2(p_2 - 1) \\
			&= \mathrm{O}(n^{1/4})\mathrm{O}(n^{1/4}) + \mathrm{O}(n^{1/2}) = \mathrm{O}(n^{1/2}) = \mathrm{o}(n).
		\end{align*}
	\end{proof}

	
	\subsection{Dehn multi twists}
	
	We will use Dehn multi twists along specific rational directions to prove irreducibility. Assume that there exist rational directions $(p_r, q_r)$ for $0 \leq r < d$ such that:
	
	\begin{enumerate}
		\item the cylinder decomposition along $(p_r, q_r)$ consists of eight cylinders with waist curves $c_g^r$, for $g \in Q$, of the same length. Thus, the Dehn multi twist along $(p_r, q_r)$ can be written as $T_r v = v + n_r \sum_{g \in G} \langle v, c_g^r\rangle c_g^r$; and \label{cond1}
		\item the action of $G$ on the labels is ``well-behaved'', that is, $(\varphi_h)_\ast c_g^r = c_{hg}^r$ for every $0 \leq r < d$, and $g, h \in Q$. \label{cond2}
	\end{enumerate}
	
	Let $Q^+ = \{1,i,j,k\}$ and $\hat{c}_g^r = c_g^r - c_{-g}^r$ for each $g \in Q^+$. If $v \in H_1^-(\widetilde{X}^{(d)})$ we have that
	\begin{align*}
		\langle v, \hat{c}_g^r\rangle\hat{c}_g^r &= \langle v, c_g^r - c_{-g}^r\rangle (c_g^r - c_{-g}^r) \\
		&= \langle v, c_g^r\rangle c_g^r - \langle v, c_{-g}^r\rangle c_g^r - \langle v, c_g^r\rangle c_{-g}^r + \langle v, c_{-g}^r\rangle c_{-g}^r \\
		&= \langle v, c_g^r\rangle c_g^r + \langle v, c_g^r\rangle c_g^r + \langle v, c_{-g}^r\rangle c_{-g}^r + \langle v, c_{-g}^r\rangle c_{-g}^r \\
		&= 2(\langle v, c_g^r\rangle c_g^r + \langle v, c_{-g}^r\rangle c_{-g}^r),
	\end{align*}
	where we used that $(\varphi_{-1})_\ast v = -v$ and that $(\varphi_{-1})_\ast$ is a symplectic automorphism. Therefore, $T_r v = v + \frac{n_r}{2} \sum_{g \in Q^+} \langle v, \hat{c}_g^r\rangle \hat{c}_g^r$. Let $C_r = \langle \hat{c}_g^r \rangle_{g \in Q^+}$. We will also assume the following:
	\begin{enumerate}\setcounter{enumi}{2}
		\item $C = \{\hat{c}_g^r\}_{g,r}$ is a basis of $H_1^-(\widetilde{X}^{(d)})$; \label{cond3}
		\item for each $0 \leq r, s < d$ and $v \in C_r \setminus \{0\}$, $T_s v \neq v$; and \label{cond4}
		\item for any $v \in C_0 \setminus \{0\}$, $C_0 = \langle\{v\} \cup \{(T_0 - \Id)(T_r - \Id)(T_1 - \Id)v \}_{r = 2}^4\rangle$. \label{cond5}
	\end{enumerate}
	
	These conditions are enough to prove that $H_1^-(\widetilde{X}^{(d)})$ is strongly irreducible for the action of $\Aff_*(\widetilde{X}^{(d)})$:
	\begin{lem}
		Assume that \eqref{cond1}--\eqref{cond5} hold. Then, $\Aff_{\ast\ast}(\widetilde{X}^{(d)})$ acts irreducibly on $H_1^-(\widetilde{X}^{(d)})$, where $\Aff_{\ast\ast}(\widetilde{X}^{(d)})$ is any finite-index subgroup of $\Aff_{\ast}(\widetilde{X}^{(d)})$.
	\end{lem}
	\begin{proof}
		Let $V \neq \{0\}$ be a subspace of $H_1^-(\widetilde{X}^{(d)})$ on which $\Aff_{\ast\ast}(\widetilde{X}^{(d)})$ acts irreducibly. By \eqref{cond3}, it is enough to prove that $C_r \subseteq V$ for each $0 \leq r < d$.
		
		Since the index of $\Aff_{\ast\ast}(\widetilde{X}^{(d)})$ is finite, some power of $T_r$ belongs to $\Aff_{\ast\ast}(\widetilde{X}^{(d)})$ for every $0 \leq r < d$. Without loss of generality, we can assume $T_r \in \Aff_{\ast\ast}(\widetilde{X}^{(d)})$, since the number $n_r \neq 0$ in the formula for $T_r$ is irrelevant for the proof.
		
		We will first show that $C_0 \subseteq V$. Let $u \in V \setminus \{0\}$. Since $H_1^-(\widetilde{X}^{(d)})$ is symplectic, by \eqref{cond3} there exists $0 \leq r < d$ such that $T_r(u) \neq u$. Clearly, $w = (T_r - \Id)u \in C_r \setminus \{0\}$. Now, by \eqref{cond4}, $v = (T_0 - \Id)w \in C_0 \setminus \{0\}$. Finally, by \eqref{cond5} we have that $C_0 \subseteq V$.
		
		Now, it is enough to show that $(T_r - \Id)C_0 = C_r$. Let $v = (T_r - \Id)\hat{c}_1^0 \in C_r \setminus \{0\}$. Observe that \eqref{cond2} implies that $C_0$ is $G$-invariant. Since $\Aff_{\ast\ast}(\widetilde{X}^{(d)})$ commutes with $G$, we have that $(\varphi_g)_\ast v \in V$ for each $g \in Q^+$. Write $v = \sum_{g \in Q^+} \mu_g \hat{c}_g^r$. By \eqref{cond2}, we have that
				
		\begin{align*}
			(\varphi_i)_\ast v &= -\mu_i \hat{c}_1^r + \mu_1 \hat{c}_i^r - \mu_k \hat{c}_j^r + \mu_j \hat{c}_k^r \\
			(\varphi_j)_\ast v &= -\mu_j \hat{c}_1^r + \mu_k \hat{c}_i^r + \mu_1 \hat{c}_j^r - \mu_i \hat{c}_k^r \\
			(\varphi_k)_\ast v &= -\mu_k \hat{c}_1^r - \mu_j \hat{c}_i^r + \mu_i \hat{c}_j^r + \mu_1 \hat{c}_k^r.
		\end{align*}
		
		Therefore, the matrix of coefficients of $(\varphi_g)_\ast v$ for $g \in Q^+$ is
		\[
			\begin{pmatrix}
				\mu_1 & \mu_i & \mu_j & \mu_k \\ 
				-\mu_i & \mu_1 & -\mu_k & \mu_j \\
				-\mu_j & \mu_k & \mu_1 & -\mu_i \\
				-\mu_k & -\mu_j & \mu_i & \mu_1
			\end{pmatrix}
		\]
		whose determinant is $\left(\sum_{g \in Q^+} \mu_g^2 \right)^2 \neq 0$. We obtain that $\langle (\varphi_g)_\ast v\rangle_{g \in Q^+} = C_r \subseteq V$, which completes the proof.
	\end{proof}
	
	We can now show that this conditions are enough for the monodromy group to be $\SO^*(2d)$:
		
	\begin{prop}
		Assume that $d = 3 \bmod 8$, that $d \in \D$ and that \eqref{cond1}--\eqref{cond5} hold. Then, the Zarisk-closure of the group $\rho(\Aff_*(X))|_{H_1^-(\widetilde{X}^{(d)})}$ is $\SO^*(2d)$.
	\end{prop}
	
	\begin{proof}
		By \Cref{lem:zero_exp}, exactly four Lyapunov exponents of the Kontsevich--Zorich cocycle are zero, so the hypotheses of \Cref{thm:zero_lyap_constraints} are satisfied. To conclude by \Cref{lem:dim_constraints}, it is enough for $\Aff_*(X)$ to act strongly irreducibly on $H_1^-(\widetilde{X}^{(d)})$, which follows from the previous lemma.
	\end{proof}
	
	The next section is then devoted to finding the desired Dehn multi twists.
	
	\subsection{Suitable rational directions} In this section, we will find the desired rational directions $(p_r, q_r)$ and prove \eqref{cond1}--\eqref{cond5} for the specific values of $d$ mentioned in the statement of the main theorem to conclude the proof. Assume that $d = 3 \bmod 8$ for the rest of the section.
	
	The matrices
	\[
		T = \begin{pmatrix}
			1 & 1 \\
			0 & 1
		\end{pmatrix}, \quad
		S = \begin{pmatrix}
			1 & 0 \\
			1 & 1
		\end{pmatrix}
	\]
	generate $\SL(2,\ZZ)$ and, thus, can be used to understand the $\SL(2,\ZZ)$-orbit of a square-tiled surface. The orbit of the ``staircase'' $X^{(d)}$ consists of three elements, which we call $Z^{(d)}$, $X^{(d)}$ and $Y^{(d)}$. See \Cref{fig:stairs_graph}.
	
	\begin{figure}
		\includegraphics[width=\textwidth]{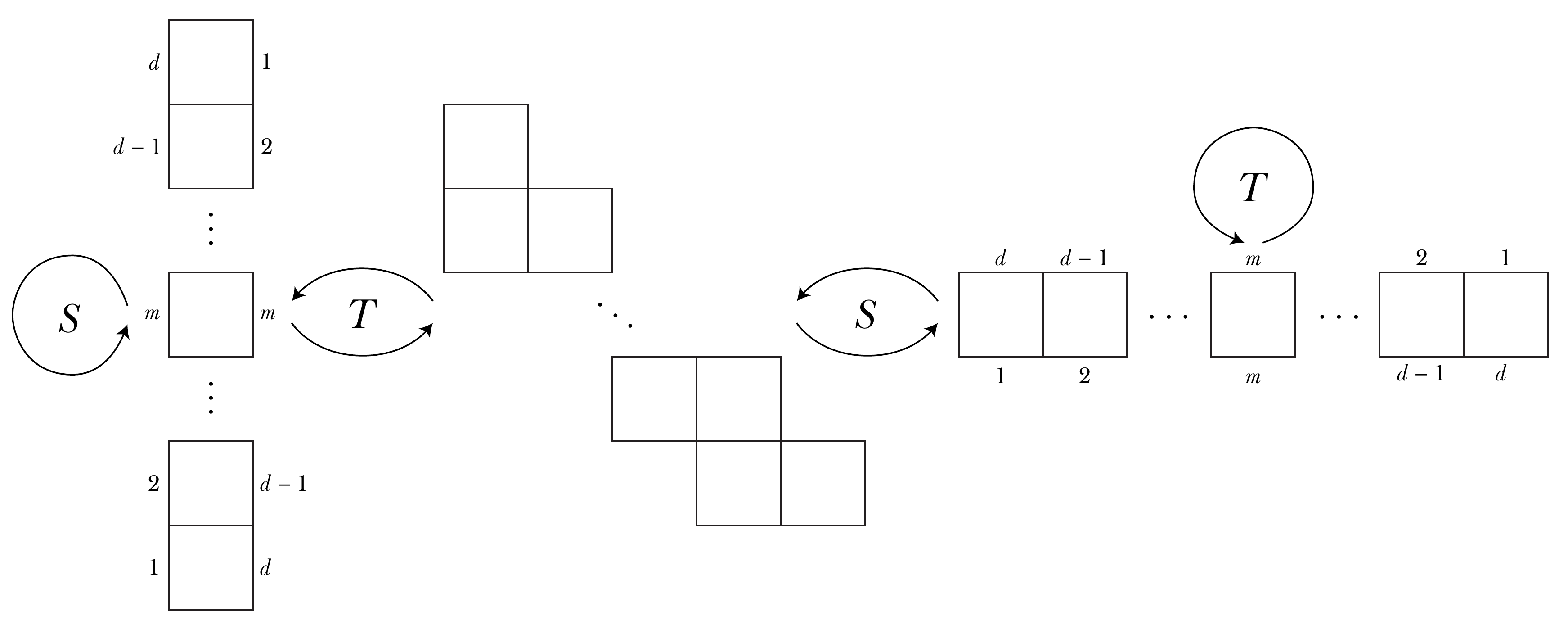}
		\caption{The $\SL(2,\ZZ)$-orbit of $X^{(d)}$ using $T$ and $S$ as generators. It consists of three distinct square-tiled surfaces, which we call $Z^{(d)}$, $X^{(d)}$ and $Y^{(d)}$ from left to right. The labels in the $Y^{(d)}$ and $Z^{(d)}$ show the identification of the sides. Unlabelled horizontal sides are identified with the only horizontal having the same horizontal coordinates, and similarly for unlabelled vertical sides.}
		\label{fig:stairs_graph}
	\end{figure}
	
	Since $\widetilde{X}^{(d)}$ is a cover of $X^{(d)}$, the graph induced by the action on $T$ and $S$ on $\widetilde{X}^{(d)}$ is a cover of the graph in \Cref{fig:stairs_graph}. In other words, if $g \in \SL(2, \ZZ)$ then $g \cdot \widetilde{X}^{(d)}$ is a degree-eight cover of $g \cdot X^{(d)}$. Moreover, since the graph in \Cref{fig:stairs_graph} has only three vertices, writing $g$ in terms of $T$ and $S$ and following the arrows of the graph  us to compute $g \cdot X^{(d)}$, which is useful to understand $g \cdot \widetilde{X}^{(d)}$.
	
	We will use the following rational directions: $(p_r, q_r) = (-(4r + 1), 4r + 3)$ for $0 \leq r < d$.
	
	Observe that the matrix $\begin{pmatrix}
		2r+1 & 2r \\ 4r+3 & 4r+1
	\end{pmatrix}$ maps the direction $(p_r, q_r)$ to $(-1,0)$. Moreover, this matrix can be written as $S^2 T^{2r} S$. By \Cref{fig:stairs_graph}, $S^2 T^{2r}S \cdot X^{(d)} = Y^{(d)}$, so this surface has only one horizontal cylinder.
	
	The matrix $S$ maps $(p_r, q_r)$ to $(-(4r+1), 2)$. The surface $S\cdot \widetilde{X}^{(d)}$, which we call $\widetilde{Y}^{(d)}$, is a degree-eight cover of $S\cdot X^{(d)} = Y^{(d)}$, which we will now describe explicitly.
	
	For each $g \in Q$, consider a copy $Y_g^{(d)}$ of $Y^{(d)}$. Each of these copies consists of $d$ squares. We label the $r$-th bottom side of each square of $Y_g^{(d)}$ with $\eta_g^r$ and the left side of the leftmost square with $\zeta_g$.
	
	Let $m = (d+1)/2$, which satisfies $m = 2 \bmod 4$ since $d = 3 \bmod 8$. There are $m-1$ squares to the left and to the right of $m$ in $Y_g^{(d)}$. The labels of the top sides of the squares to the right of $m$ are:
	\[
		\eta_{gk}^{m-1}, \eta_{-g}^{m-2}, \eta_{-gk}^{m-3}, \eta_g^{m-4}, \dotsc, \eta_{gk}^{5}, \eta_{-g}^{4}, \eta_{-gk}^{3}, \eta_g^{2}, \eta_{gk}^1.
	\]
	The labels of the top sides of the squares to the left of $m$ are:
	\[
		\eta_{-gi}^d, \eta_{gj}^{d-1}, \eta_{gi}^{d-2}, \eta_{-gj}^{d-3}, \eta_{-gi}^{d-4}, \dotsc, \eta_{gi}^{m+3}, \eta_{-gj}^{m+2}, \eta_{-gi}^{m+1}, \eta_{gj}^{m}.
	\]
	In the two previous lists, the group elements in $Q$ follow a $4$-periodic pattern. Finally, we label the rightmost square of $Y_g^{(d)}$ with $\zeta_{-g}$. See \Cref{fig:horiz1} for an illustration.
	
	\begin{figure}
		\includegraphics[width=\textwidth]{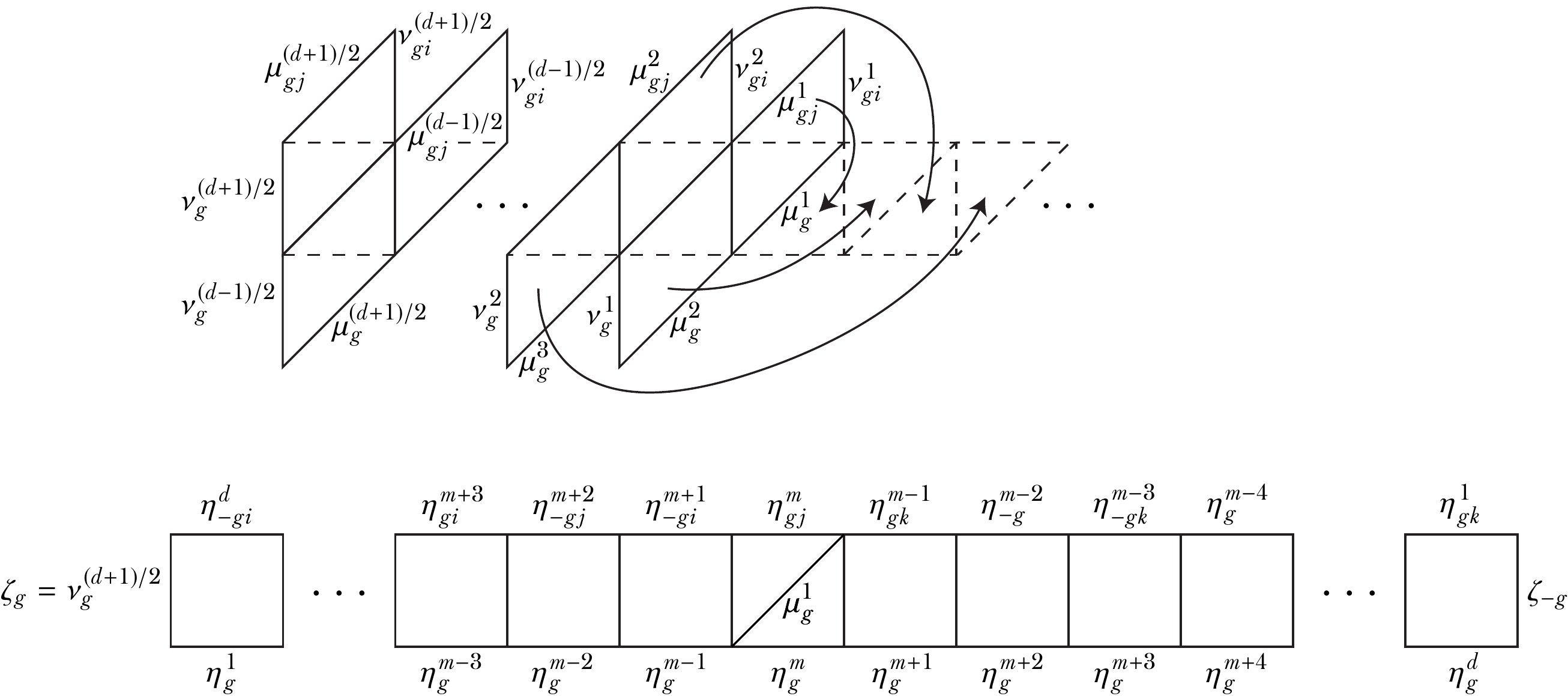}
		\caption{An illustration of $Y_g^{(d)}$ and of the cut-and-paste operations used to obtain this description.}
		\label{fig:horiz1}
	\end{figure}
	
	By a slight abuse of notation, from now on we will use the names $\eta_g^r$ and $\zeta_g$ to refer to the elements of $H_1(\widetilde{Y}^{(d)}, \Sigma; \RR)$ induced by the horizontal or vertical curves joining the two vertices of the side labelled $\eta_g^r$ or $\zeta_g$, oriented either rightwards or upwards. We have that $\Aut(\widetilde{Y}^{(d)}) \simeq Q$, which can be proved in the exact same way as for $\widetilde{X}^{(d)}$. That is, we define an automorphism $\varphi_g$ by mapping $Y_h^{(d)}$ to $Y_{gh}^{(d)}$ and these are the only automorphisms of $\widetilde{Y}^{(d)}$ since $\Aut(Y^{(d)})$ is trivial. The automorphism $\varphi_{-1}$ induces a splitting $H_1(\widetilde{Y}^{(d)}; \RR) = H^+(\widetilde{Y}^{(d)}) \oplus H^-(\widetilde{Y}^{(d)})$. The space $H^-(\widetilde{Y}^{(d)})$ is $4d$-dimensional and it is exactly the image of $H^-(\widetilde{X}^{(d)})$ by $S$. Let $\hat{\eta}_g^r = \eta_g^r - \eta_{-g}^r$ for $g \in Q^+ = \{1, i, j, k\}$ and $1 \leq r \leq d$. We have that each $\hat{\eta}_g^r$ is an absolute cycle since $\varphi_{-1}$ fixes every singularity. Therefore, $\hat{\eta}_g^r \in H_1^-(\widetilde{Y}^{(d)})$ and we obtain that $\{\hat{\eta}_g^r\}_{g \in Q^+, 1 \leq r \leq d}$ is a basis of $H_1^-(\widetilde{Y}^{(d)})$.
	
	Observe that $\widetilde{Y}^{(d)}$ has four horizontal cylinders. The matrix $T^{2r}$ maps the direction $S(p_r,q_r) = (-(4r+1),2)$ to $(-1,2)$. Therefore, understanding the direction $(p_r,q_r)$ on $\widetilde{X}^{(d)}$ is equivalent to understanding the direction $(-1,2)$ on $T^{2r} \cdot \widetilde{Y}^{(d)}$.
	
	\begin{figure}
		\includegraphics[width=\textwidth]{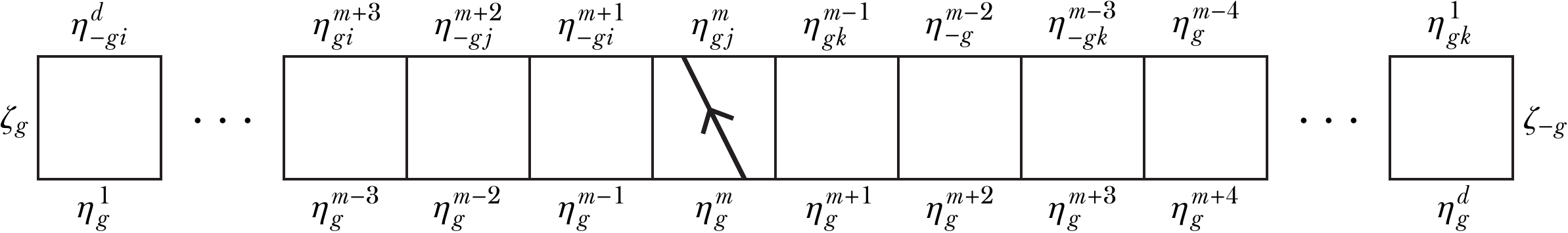}
		\caption{Direction $(-1,2)$ on $Y_g^{(d)}$.}
		\label{fig:direction}
	\end{figure}
	
	We will start the analysis for $r = 0$. For $g \in Q$, consider the trajectory induced by the direction $(-1,2)$ on $Y_g^{(d)}$ as in \Cref{fig:direction}. The resulting cylinder decomposition consists on eight cylinders. Indeed, observe that each cycle $\eta_g^r$ is intersected twice by such trajectories. Therefore, the total number of intersections of all the $\eta_g^r$ by all trajectories is $16d$. To obtain that there are exactly eight cylinders in this decomposition, it is therefore enough to show that each trajectory intersects exactly $2d$ cycles $\eta_g^r$.
	
	The trajectory in \Cref{fig:direction} intersects the following cycles:
	\begin{align*}
		&\eta_{g_1}^m, \eta_{g_2}^{m+1}, \eta_{g_3}^{m-1}, \eta_{g_4}^{m+2}, \eta_{g_5}^{m-2}, \dotsc, \eta_{g_{d-1}}^d, \eta_{g_d}^1, \zeta_{g_d}, \\
		&\eta_{g_{d+2}}^1, \eta_{g_{d+3}}^d, \dotsc, \eta_{g_{2d-3}}^{m-2}, \eta_{g_{2d-2}}^{m+2}, \eta_{g_{2d-1}}^{m-1}, \eta_{g_{2d}}^{m+1}, \eta_{g_{2d+1}}^m
	\end{align*}
	where the sequence $g_1, \dotsc, g_{2d+1}$ is obtained by (right-)multiplying $g$ successively by
	\begin{equation}\label{eq:prod}
		j, -i, k, -j, -1, i, -k \dotsc, j, 1, -i, k, \boxed{-1}, k, -i, 1, j, \dotsc, -k, i, -1, -j, k, -i, j. \tag{$\ast$}
	\end{equation}
	The boxed $-1$ comes from the intersection with the vertical side labelled as $\zeta_{g_d}$. 
	
	This sequence indeed describes a closed trajectory as $g_{2d+1} = g$. Indeed, the product can be computed from ``inside out'', using that $-1$ is in the centre of $Q$. We obtain that $g_{2d+1} = -g \cdot k^2 \cdot (-i)^2 \cdot 1^2 \cdot j^2 \dotsb (-i)^2 j^2$. Moreover, the number of times $1^2$ and $(-1)^2$ occur in this product is $(m-2)/2$, which is an even number as $m \bmod 4 = 2$, and the total number of terms is $d+1$, which is also an even number. Thus, $g_{2d+1} = g$ and we conclude that the cylinder decomposition induced by $(-1,2)$ has exactly eight cylinders. Moreover, we obtain that the action of $\Aut(\widetilde{Y}^{(d)})$ on these waist curves is ``well-behaved'' in the sense of \eqref{cond2}: naming the trajectory starting on $Y_g^{(d)}$ as $c_g^0$, we get that $(\varphi_{h})_\ast c_g^0 = c_{hg}^0$. 
	
	Now, if $r = 1$ then $Y_g^{(d)}$ is sheared horizontally in such a way that the labels $\eta_g^{m+1}$ and $\eta_{-gi}^{m+1}$ end up on the same square. We will consider this square to be the ``middle'' square and reglue the surface accordingly. The surface $T^2 \cdot \widetilde{Y}^{(d)}$ is the union of sheared and reglued versions of $Y_g^{(d)}$, for $g \in Q$, that we call $T^2 \cdot Y_g^{(d)}$. See \Cref{fig:T2_S_stairs_orbit} for an illustration.
	\begin{figure}
		\includegraphics[width=\textwidth]{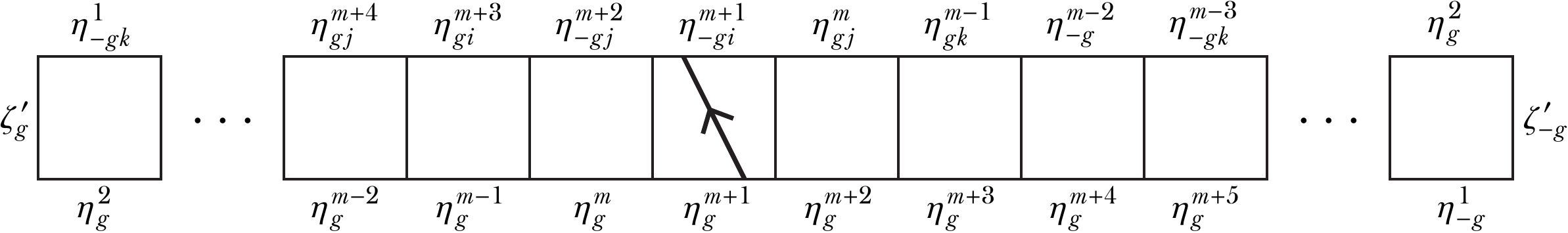}
		\caption{Direction $(-1,2)$ on $T^2 \cdot Y_g^{(d)}$. The gluings are cyclically shifted and the signs of elements of $Q$ on the labels $\eta_\bullet^1$ are changed.}
		\label{fig:T2_S_stairs_orbit}
	\end{figure}
	
	In general, $T^{2r} \cdot \widetilde{Y}^{(d)}$, for $0 \leq r < d$, is the surface obtained from $\widetilde{Y}^{(d)}$ by cyclically shifting the labels on the top sides $r$ times to the right, the ones on the bottom sides $r$ times to the left, and changing the signs of the elements of $Q$ of every label of the form $\eta_\bullet^s$ for $1 \leq s \leq r$. We conclude that the cylinder decomposition of $T^{2r} \cdot \widetilde{Y}^{(d)}$ induced by the direction $(-1, 2)$ consists of exactly eight cylinders in the same way as for the case $r = 0$ and denote their waist curves by $c_g^r$. The action of $G$ is then well-behaved in the sense of \eqref{cond2}. By construction, \eqref{cond1} also holds.
	
	Let $\hat{c}_g^r = c_g^r - c_{-g}^r$ for $g \in Q^+$. It remains to prove \eqref{cond3}, \eqref{cond4} and \eqref{cond5} to conclude the proof. We conjecture that these two conditions hold for every $d$ belonging to the congruence class $d = 3 \mod 8$.
	
	Nevertheless, the previous discussion allows us to compute the intersection numbers explicitly using a computer. Indeed, to obtain \eqref{cond3} we can compute the numbers $\langle \hat{c}_g^r, \hat{\eta}_h^s\rangle$ for each $0 \leq r, s \leq d$. Then, we can compute the determinant of the resulting matrix to show that it is not singular. This matrix also allows us to compute $\langle \hat{c}_g^r, \hat{c}_h^s \rangle$ by expressing each $\hat{c}_g^r$ in terms of the basis $\{\hat{\eta}_g^r\}_{g \in Q^+, 1 \leq r \leq d}$ to show \eqref{cond4} and \eqref{cond5}. The computations were done for $11 \leq d \leq 299$. Observe that $d = 35$ and $d = 203$ are the only elements of $\{11, \dotsc, 299\}$ satisfying $d = 3 \bmod 8$ and not belonging to $\D$, as $2 \times 35 = \binom{8}{4}$ and $2 \times 203 = \binom{29}{2}$. In this way, \Cref{thm:main} is proved.

	\medbreak
	\textbf{Acknowledgements:} I am grateful to Simion Filip and Nicolás Sanhueza-Matamala for useful discussions. I am also grateful to my advisors, Anton Zorich and Carlos Matheus.

	\sloppy\printbibliography
	
\end{document}